\documentclass[12pt]{article}

\usepackage[english] {babel}
\usepackage{amsmath,amssymb,amscd}
\usepackage{amsthm}
\usepackage{verbatim}
\usepackage[utf8]{inputenc}
\usepackage{graphicx}
\usepackage{fancybox}
\usepackage{subfigure}
\usepackage[usenames,dvipsnames]{color}
\usepackage{sidecap}
\usepackage{caption}
\usepackage{ifthen}
\usepackage{textcomp}
\usepackage[all]{xy}
\usepackage{feynmf}
\usepackage{anysize}

\newcommand{\N}{{\mathbb N}}

\newcommand{\Z}{{\mathbb Z}}
\newcommand{\R}{{\mathbb R}}
\newcommand{\C}{{\mathbb C}}
\newcommand{\Q}{{\mathbb Q}}

\newcommand{\fdem}{\hfill $\square$}

\theoremstyle{plain}
\newtheorem{teo}{Theorem}[section]
\newtheorem{lema}[teo]{Lemma}
\newtheorem{cor}[teo]{Corollary}
\newtheorem{prop}[teo]{Proposition}

\theoremstyle{definition}
\newtheorem{defi}{Definition}[section]

\theoremstyle{remark}
\newtheorem{obs}{Remark}[section]

\def\noi{\noindent}

\begin{document}

\nocite{*}

\title{Adelic solenoid II: Ahlfors-Bers theory}

\author{Juan M. Burgos, Alberto Verjovsky \thanks{The second author (AV) benefited from a PAPIIT (DGAPA, Universidad Nacional Aut\'onoma de M\'exico) grant IN106817.}}

\maketitle

\noi {\bf 2010 Mathematics Subject Classification.} Primary: 11R56, 32G05, 32G15, 57R30.
\noi Secondary: 22Bxx, 37J40.


\begin{abstract}
\noi We generalize the Ahlfors-Bers theory to the adelic Riemann sphere. In particular, after defining the appropriate notion of a Beltrami differential in the solenoidal context, we give a sufficient condition on it such that the corresponding Beltrami equation has a quasiconformal homeomorphism solution; i.e. The Ahlfors-Bers Theorem in the solenoidal case. This additional condition on the solenoidal Beltrami differentials can be written as a Banach norm in a subspace of solenoidal differentials. Moreover, this subspace is the completion under this norm of those solenoidal differentials locally constant at the fiber. As a toy example, we show how this technique works on a linear problem: We generalize the diophantine equation complex analytic extension problem to the respective solenoidal space.
\end{abstract}

\section{Introduction}

In the last thirty years, dynamical systems, differential equations on ultrametric spaces (p-adics, Berkovich spaces, ad\`eles) and their  applications have been studied intensively, see e.g. \cite{1}, \cite{2}, \cite{3}, \cite{4}, \cite{5}, \cite{6}, \cite{7}, \cite{8}, \cite{9}, \cite{10}, \cite{11} and the references therein.

In this paper we study the Beltrami differential equation on the adelic Riemann sphere. It is the inverse limit of the inverse system of coverings $z\mapsto z^{n}$ on the Riemann sphere ramifified at $0$ and $\infty$. Topologically, it is the suspension of the adelic solenoid. Removing the cusps, this is a solenoidal laminated object with nontrivial monodromy, fibering over the punctured plane whose typical fiber is the ring of adelic integers.

Ahlfors-Bers theory is the complex analytic theory of the Teichm\"uller space \cite{Ahlfors},\cite{Hubbard},\cite{Imayoshi}. This is the space of equivalence classes of complex structures on a Riemann surface where two structures are equivalent if they are related by a homeomoprhism homotopic to the identity. In Ahlfors-Bers theory, complex structures are obtained by deformations of a prescribed structure (a base point in Teichm\"uller space) by quasiconformal homeomorphisms. After uniformization of the Riemann surface on the Poincar\'e disk $\Delta$, these maps are obtained by a Beltrami coefficient $\mu$ in the unit ball $L_\infty(\Delta)_1$ and a normalization condition through the Beltrami equation:
$$\partial_{\bar{z}}f= \mu\partial_z f$$

At the heart of the theory lies the Ahlfors-Bers Theorem which states that the Beltrami equation admits a quasiconformal homeomorphism solution. Considering that the Beltrami equation is distributional and the coefficient is just an $L_\infty$ measurable class, this is a striking result. In particular by considering the Beltrami coefficients, this gives a complex analytic model for the universal Teichm\"uller space:
$$T(1):= L_\infty(\Delta)_1/\sim$$
where two coefficients are equivalent if the quasiconformal solutions of the respective Beltrami equations coincide on the the boundary of the disk.

Although we described the Beltrami equation in the context of Teichm\"uller theory, actually the coefficient lives in $L_\infty(\C)_1$ and the solution is a quasiconformal homeomorphism of the Riemann sphere $\C P^{1}$. Here we generalize the Ahlfors-Bers Theorem to the adelic Riemann sphere. For this purpose we must develop the appropriate notions of solenoidal Beltrami differential, solenoidal Beltrami equation and solenoidal quasiconformal solution.

However, to assure the existence of a solution, we will need an additional sufficient condition on the solenoidal Beltrami differentials. This condition can be interpreted as a finer Banach norm on a subspace of solenoidal differentials. This subspace is the completion under the new norm of those solenoidal differentials locally constant at the fiber. In particular, we have a solenoidal Ahlfors-Theorem for those new solenoidal Beltrami differentials living strictly in the completion.

After establishing this result, the construction of the corresponding complex analytic Teichm\"uller space is completely straightforward and we do not even mention it in the text. However, it is worth to point out that the real version of it has been considered before in dynamical systems theory: In  \cite{Universalities} (see also \cite{Su2}, section $9$), D. Sullivan studies the linking between universalities of Milnor-Thurston, Feigenbaum's (quantitative) and Ahlfors-Bers. As he points out in his second example, the dynamical suspension $L$ of the square map from the diadic solenoid $S^{1}_{2}$ to itself is the basic solenoidal surface required in the dynamical theory of Feigenbaum's Universality \cite{Feigenbaum}. This object can be seen as a quotient of $\Delta^{*}_\Q$, the inverse limit of the punctured disk by the inverse system of coverings considered before. This is a two dimensional solenoid with hyperbolic leaves and so is $L$. In that paper (see also \cite{Su2}, appendix), it is developed the real Teichm\"uller theory of a two dimensional solenoid and then it is applied to the dynamical solenoid $L$. We can say that the present paper develops the complex analytic theory for the Teichm\"uller space of $L$.

As a toy example and prelude to the solenoidal Beltrami equation, we solve a linear problem: We generalize the diophantine equation complex analytic extension problem in the space $\C^{n}/\Z^{n}\cong (\C^{*})^{n}$ to the solenoidal space $(\C_\Q^{*})^{n}$ where $\C_\Q^{*}$ is the algebraic solenoid. We solve the problem by imposing a finer Banach norm on a subspace of the space of complex analytic functions where the non homogeneous term of the equation lives.


\section{Toy example: The diophantine equation}

\subsection{Introduction}

After reading part one, the reader may have had the impression that adapting the usual concepts and results to the solenoidal case is an always possible and straightforward procedure. Even he might have guessed a sort of metaprinciple of writing the rational numbers $\Q$ instead of the integers $\Z$ in order to translate the statements to the solenoidal case.

The purpose of this section is dual. The first is to show that the previous thought is false in the case of differential equations. We will generalize the problem of extending complex analitically a solution of a diophantine equation in the space $\C^{n}/\Z^{n}\cong (\C^{*})^{n}$ to the solenoidal space $(\C_\Q^{*})^{n}$. Here $\C^{*}_\Q$ is the inverse limit of the inverse system of coverings $z\mapsto z^{n}$ on the punctured plane $\C^{*}$. Topologically it is homeomorphic to the adelic solenoid times the real line.

Instead of working on the solenoid we will do it over its baseleaf for it is equivalent and formally easier. The extension problem for diphantine equations is the following: Given a zero average $2\pi$--periodic (respect to the real variables) complex analytic function $g$ in $\mathcal{C}_\rho$ and a vector $\omega$ in $\R^{n}$, find a complex analytic solution $f$ possibly with smaller domain in $\mathcal{C}_{\rho-\delta}$ with $\delta>0$, of the diophantine equation $D f(\omega)= g$. Here $\mathcal{C}_\rho$ denotes the space of complex analytic $2\pi$--periodic functions whose supremum norm over $||Im(z)||<\rho$ with $z\in\C^{n}$ is finite. This problem will not have solution for general $\omega$ due to the \textsf{small divisor problem}: Translating the equation into Fourier space, the Fourier coefficients will verify the following relation:
$$f_k:= \frac{g_k}{2\pi i(\omega\cdot k)}, \quad k\in \Z^{n}, \quad k\neq 0$$
The problem is that, for general $\omega$, the divisor could get arbitrarily small or even zero. However, there is a special class of vectors $\omega$ for which this problem can be controlled giving a convergent Fourier series solution on some domain. These are the \textsf{diophantine vectors} and verify the following estimate:
$$|\omega\cdot k|>\frac{\gamma}{|k|^{n}}, \quad k\in \Z^{n}, \quad k\neq 0$$
for some positive constant $\gamma$. This condition is the generalization of diophantine irrationals of degree three. The set of diophantine vectors is of full measure in $\R^{n}$.

We generalize this problem to the solenoidal case. Now we have Pontryagin series instead of Fourier and the modes are rational nuebers. Seeking for a solution, note that invoking the metaprinciple mentioned at the beginning, there is no vector $\omega$ verifying the diophantine condition if we consider $\Q$ instead of $\Z$.

Now, the space to consider is $\mathcal{C}_{\rho, \Q}$: This is the space of complex analytic limit periodic functions whose supremum norm over $||Im(z)||<\rho$ with $z\in\C^{n}$ is finite. We will solve the problem by imposing a finer Banach norm on a subspace of $\mathcal{C}_\rho$ where the function $g$ lives. However, the space $\mathcal{C}_{\rho-\delta}$ where the solution lives has the original supremum norm.

The proof of the KAM Theorem at least for the simplest case of an analytic nondegenerate Hamiltionian in \cite{Benettin} and \cite{Hubbard2}, goes through an iterative process solving a diophantine equation in each step. Because of the fact that the new norm is finer than the original supremum one, this iterative process cannot be applied directly to the solenoidal case and unfortunately we do not have a solenoidal KAM Theorem following this procedure.

The second purpose of this section is to be a prelude for the solenoidal Beltrami equation. In constrast to the diophantine equation, the solenoidal Beltrami equation is nonlinear and the solution will be a leaf preserving homeomorphism from the adelic Riemann sphere to itself instead of a complex valued function. This will be a much harder problem and the techniques involved will be quite different. However, they have similarities: Although the general case admits no solution, imposing an additional condition on the input data (the Beltrami differential in this case) will guarantee a solution. This additional condition is in escence the same as the one considered here imposing the new norm and they look very similar indeed.

\subsection{Diophantine equation}

Consider the following supremum norm on the space of holomorphic functions on $\C^{n}$:
$$||f||_\rho:= \sup_{||Im(z)||<\rho} |f(z)|$$
Define the space $\mathcal{C}_{\rho, L}$ of $2\pi L$--periodic holomorphic functions $f$ such that $||f||_{\rho}<\infty$. The space $\mathcal{C}_{\rho, L}$ is a Banach space with the supremum norm $||\cdot||_\rho$. Define the space $\mathcal{C}_{\rho, \Q}$ verbatim as before with limit periodic functions instead of periodic. Again, this space is a Banach space under the supremum norm.

Define the set $\Omega_\gamma\subset \R^{n}$ of \textsf{diophantine vectors} $\omega$ respect to some positive constant $\gamma$ verifying the property:
$$|\omega\cdot k|>\frac{\gamma}{|k|^{n}}$$
for every $k\in \Z^{n}-\{0\}$. The union $\bigcup_{\gamma>0} \Omega_\gamma$ is of full measure (\cite{Hubbard2}, Proposition 5.3). We have the following Theorem (\cite{Hubbard2}, Proposition 7.3):

\begin{teo}\label{Preliminary_bound}
Consider a zero average (respect to the real variables) function $g\in \mathcal{C}_{\rho, 1}$ and a diophantine vector $\omega\in \Omega_{\gamma}$ for some positive constant $\gamma$. Then, for every $\delta>0$ there is a solution $f\in \mathcal{C}_{\rho-\delta,1}$ of the corresponding diophantine equation such that:
$$||f||_{\rho-\delta}\leq \frac{k_n}{\gamma\delta^{2n}} ||g||_\rho\quad and \quad||Df||_{\rho-\delta}\leq \frac{k_n}{\gamma\delta^{2n+1}} ||g||_\rho $$
\end{teo}

We want to generalize the previous Theorem to the space $\mathcal{C}_{\rho, \Q}$; i.e. The solenoid context. For this purpose, we need the rescale version of it:
\begin{lema}\label{rescale}
Consider a zero average (respect to the real variables) $g\in \mathcal{C}_{\rho, L}$ and a diophantine vector $\omega\in \Omega_{\gamma}$ for some positive constant $\gamma$. Then, for every $\delta>0$ there is a solution $f\in \mathcal{C}_{\rho-\delta, L}$ of the corresponding diophantine equation such that:
$$||f||_{\rho-\delta}\leq \frac{k_n L^{2n+1}}{\gamma\delta^{2n}} ||g||_\rho\quad and \quad ||Df||_{\rho-\delta}\leq \frac{k_n L^{2n+1}}{\gamma\delta^{2n+1}} ||g||_\rho$$
\end{lema}
\begin{proof}
Consider the homothety $h_a$ in $\C^{n}$ such that $h_L(z_1.\ldots z_n):= a (z_1,\ldots z_n)$. Now, the function $g\circ h_L$ belongs to $\mathcal{C}_{L^{-1}\rho,1}$ and the new diophantine equation has diophantine vector $\tilde{\omega}= L^{-1}\omega\in \Omega_{L^{-1}\gamma}$.

Then, by Lemma \ref{Preliminary_bound}, for every $\delta>0$ there is a solution $f_L\in \mathcal{C}_{L^{-1}(\rho-\delta), 1}$ of the corresponding diophantine equation such that:
$$||f_L||_{L^{-1}(\rho-\delta)}\leq \frac{k_n}{(L^{-1}\gamma)(L^{-1}\delta)^{2n}} ||g\circ h_L||_{L^{-1}\rho}$$
$$||Df_L||_{L^{-1}(\rho-\delta)}\leq \frac{k_n}{(L^{-1}\gamma)(L^{-1}\delta)^{2n+1}} ||g\circ h_L||_{L^{-1}\rho}$$
Define $f:= f_L\circ h_{L^{-1}}\in \mathcal{C}_{\rho-\delta, L}$. Then $Df_L= L\ Df\circ h_L$ and note that the previous estimates are equivalent to those of the claim.
\end{proof}

Define the subspace $Per$ of periodic maps in $\mathcal{C}_{\rho, \Q}$; i.e.
$$Per:= \bigcup_{L\in\N}\mathcal{C}_{\rho, L}$$
By Lemma \ref{rescale}, we have a solution to our problem for every element in $Per$. We would like to have the same situation for every element in the closure of $Per$ respect to some norm. The supremum norm will not do as the next example shows: Consider the function $g$ in $\mathcal{C}_{\rho, \Q}$ as follows:
$$g(z):= \sum_{i\in\N}\frac{1}{i!}\cos\left(2\pi \frac{z_1+\ldots z_n}{i!}\right)$$
The solution of the diophantine equation respect to some diophantine vector $\omega$ and initial condition $f(0)=0$ is the following:
$$f(z)= \frac{1}{2\pi(\omega_1+\ldots \omega_n)}\sum_{i\in\N}\sin\left(2\pi \frac{z_1+\ldots z_n}{i!}\right)$$
Although $f$ converges locally uniformly in the region $||Im(z)||<\rho$, it does not converge uniformly hence it does not belong to the space $\mathcal{C}_{\rho, \Q}$; i.e. it cannot be lifted to the solenoid.

We need to consider a finer norm. Consider a function $g$ in $\mathcal{C}_{\rho, \Q}$ and its Pontryagin series \cite{RV}:
$$g(z)=\sum_{q\in\Q}g_q e^{2\pi i q\cdot z}$$
For every natural number $L$ define the periodic holomorphic function $g_L$ as follows:
$$g_L(z)=\sum_{q\in L^{-1}\Z^{n}}a_q e^{2\pi i q\cdot z}$$
Consider a cofinal sequence $\mathcal{S}=(n_i)_{i\in\N}$ in the divisibility net such that $n_i|n_{i+1}$ for every natural number $i$ and define the following norm on $\mathcal{C}_{\rho, \Q}$:
$$||g||_{\rho, \mathcal{S},n}:= n_1^{2n+1}\parallel g_{n_1}\parallel_\rho + \sum_{i>1}\ n_i^{2n+1}\parallel g_{n_i}-g_{n_{i-1}}\parallel_\rho$$
The space $\mathcal{C}_{\rho, \Q}$ under this new norm dominates the original supremum one and defines a Banach space on the subspace of $\mathcal{C}_{\rho, \Q}$ consisting of the finite vectors respect to this new norm. We have the following Theorem:

\begin{teo}\label{Diophantine_Solenoid}
Consider $g\in \mathcal{C}_{\rho, \Q}$ such that $g_0=0$ and a diophantine vector $\omega\in \Omega_{\gamma}$ for some positive constant $\gamma$. Suppose there is a cofinal sequence $\mathcal{S}=(n_i)_{i\in\N}$ in the divisibility net such that $n_i|n_{i+1}$ for every natural number $i$ and $||g||_{\rho, \mathcal{S},n}<\infty$. Then, for every $\delta>0$ there is a solution $f\in \mathcal{C}_{\rho-\delta, \Q}$ of the corresponding diophantine equation such that:
$$||f||_{\rho-\delta}\leq \frac{k_n}{\gamma\delta^{2n}} ||g||_{\rho, \mathcal{S},n}\quad and \quad ||Df||_{\rho-\delta}\leq \frac{k_n}{\gamma\delta^{2n+1}} ||g||_{\rho, \mathcal{S},n}$$
\end{teo}

\begin{proof}
Define $f_{n_0}:=0$. The Pontryagin series of $g$ is the following:
$$g(z)=\sum_{q\in\Q}g_q e^{2\pi i q\cdot z}$$
Consider the formal solution of the diophantine equation:
$$f(z)=\sum_{\substack{q\in\Q\\ q\neq 0}} b_q e^{2\pi i q\cdot z},\quad b_q:= \frac{a_q}{2\pi i(\omega\cdot q)},\quad q\neq 0$$
Note the \textsf{small divisor problem} on the definition of $b_q$. For every natural number $L$ define the holomorphic function $g_L$ as follows:
$$g_L(z)=\sum_{q\in L^{-1}\Z^{n}}a_q e^{2\pi i q\cdot z}$$
and consider an analogous, a priori formal, definition for $f$. By Lemma \ref{rescale}, for every natural number $L$ and every $\delta>0$, $f_L\in \mathcal{C}_{\rho-\delta, L}$. Moreover, the a priori formal $f$ is limit periodic:
\begin{eqnarray*}
||f-f_{n_j}||_{\rho-\delta} = \parallel \sum_{i>j}(f_{n_i}-f_{n_{i-1}})\parallel_{\rho-\delta}\leq \sum_{i>j}\parallel f_{n_i}-f_{n_{i-1}}\parallel_{\rho-\delta} \\
\leq \frac{k_n}{\gamma\delta^{2n}}\sum_{i>j}\ n_i^{2n+1}\parallel g_{n_i}-g_{n_{i-1}}\parallel_\rho\xrightarrow{j\to \infty} 0
\end{eqnarray*}
where we have used Lemma \ref{rescale} and the fact that the diophantine equation is linear. Taking $j=0$ in the above expression we conclude that $f\in \mathcal{C}_{\rho-\delta,\Q}$ and verifies the first estimate of the claim. The other estimate follows verbatim.
\end{proof}

\section{Preliminaries on Ahlfors-Bers theory}

In this section we follow closely \cite{Imayoshi}. Consider an orientation preserving homeomorphism $f$ from a domain $D\subset\C$ into $\C$. It is \textsf{quasiconformal} if it verifies the following conditions:
\begin{enumerate}
\item The distributional partial derivatives of $f$ respect to $z$ and $\bar{z}$ can be represented by locally integrable functions $f_z$ and $f_{\bar{z}}$ respectively on $D$.
\item There exists a constant $k$ with $0\leq k<1$ such that $|f_{\bar{z}}|\leq k |f_z|$ almost everywhere on $D$.
\end{enumerate}

A \textsf{Beltrami coefficient} is an element of the unit ball $L_\infty(\C)_1$.

\begin{teo}\label{normal_solution}
Consider $k$ such that $0\leq k <1$. Then, for every Beltrami coefficient $\mu$ with $||\mu||_{\infty}\leq k$ and compact support there is a unique quasiconformal map $f:\C\rightarrow \C$ such that $f(0)=0$ and $f_{z}-1\in L_{p}(\C)$ for some $p>2$ only depending on $k$ verifying:
$$f_{\bar{z}}= \mu f_{z}$$
on $\C$ in the sense of distributions.
\end{teo}

\noi The equation in the previous Lemma is the \textsf{Beltrami equation associated to $\mu$} and its quasiconformal solution is \textsf{normal}. The constant $p$ in the previous Theorem is not unique; i.e. It is not a parameter of the solution. We will need the following lifting and composition results:

\begin{lema}
Suppose there are domains $D\subset\C$ and $\tilde{D}\subset\C$ with a covering map $\zeta:\tilde{D}\rightarrow D$. Given an orientation preserving homeomorphism $f$ of $D$ into itself, there is an orientation preserving homeomorphism $\tilde{f}$ from $\tilde{D}$ into itself such that $f\circ \zeta= \zeta\circ \tilde{f}$; i.e. A \textsf{lifting by $\zeta$}. If $f$ is a quasiconformal map and verifies the Beltrami equation associated to some Beltrami coefficient $\mu$, then idem every lifting by $\zeta$ with $\zeta^{*}(\mu)$. Moreover, if $f$ is normal then so is $\tilde{f}$. 
\end{lema}

\begin{lema}
The composition of quasiconformal normal maps is a quasiconformal normal map.
\end{lema}

We will also need the following Lemmas. Because we will apply them to a family of maps, we need to be precise on the dependence of the constants involved.

\begin{lema}\label{Cota1_Prel}
If $f$ is a normal solution of the Beltrami equation associated to $\mu$ such that $\mu$ has compact support, then there is a constants $A$ and a real number $p>2$ such that:
$$|f(\zeta)-\zeta|\leq A||\mu||_{\infty}|\zeta|^{1-2/p}$$
The constant $A$ depends only on the area of the $\mu$ support and $p$. Moreover, it is strictly increasing respect to area of the $\mu$ support.
\end{lema}

\begin{lema}\label{Cota2_prel}
If $f$ is a normal solution of the Beltrami equation associated to $\mu$ such that $\mu$ has compact support, then there is a constant $B$ and a real number $p>2$ such that:
$$|\zeta|\leq B||\mu||_{\infty}|f(\zeta)|^{1-2/p} + |f(z)|$$
The constant $B$ depends only on the area of the $\mu$ support and $p$. Moreover, it is strictly increasing respect to area of the $\mu$ support.
\end{lema}

The following is the celebrated Ahlfors-Bers Theorem.

\begin{teo}
For every Beltrami coefficient $\mu$ there is a unique quasiconformal map $f^{\mu}$ verifying the Beltrami equation associated to $\mu$ and fixing $0$, $1$ and $\infty$.
\end{teo}

\begin{prop}\label{Imayoshi_Taniguchi_Convergence_infty}
If $\mu_{n}$ tends to $\mu$ in $L_{\infty}(\C)$, then $f^{\mu_{n}}$ tends to $f^{\mu}$ locally uniformly.
\end{prop}

\section{Solenoidal Beltrami differentials}

A Beltrami differential is an $L_\infty$ section of the complex line bundle $\bar{\omega}\otimes \omega^{*}$ over $\C$ with essential supremum norm less than one. Here $\omega$ denotes the canonical line bundle. It has the expression $\mu\ d\bar{z}\otimes \partial_z$ where $\mu\in L_\infty(\C)_1$, the $L_\infty(\C)$ unit ball. It can be seen as an element in the unit ball $L_\infty(\C)_1$ which transforms in the following way:
$$\gamma^{*}\mu= \mu\circ \gamma\ \frac{\overline{\gamma_z}}{\gamma_z}$$
where $\gamma$ is a holomorphic map. We denote the unit ball in the space of $L_\infty$ sections of the complex line bundle $\bar{\omega}\otimes \omega^{*}$ over $\C$ simply as $L_\infty(\C)_1\ d\bar{z}\otimes \partial_z$. At first sight, it may seem pedantic and useless to explicitly write the differential part, actually it is commonly assumed as an abuse of notation. However, in the solenoid context we define everything in terms of the projections of the inverse system and the exponential map so it will be important, in order to keep track of the different phase contributions that will appear, to be explicit about the differentials.

We seek for an analog of Beltrami differential in the algebraic solenoid $\C^{*}_\Q$ looking forward to solve an analog of the respective Beltrami equation. Therefore, the natural way to define a Beltrami differential in the solenoid context is to go backwards; i.e. First define the Beltrami equation on the algebraic solenoid, then define what we mean by a solution of it and finally impose in the definition the obvious necessary conditions to have a solution of the respective Beltrami equation.

In what follows, $T_w$ denotes the translation map $z\mapsto z+w$ in $\C$. Consider the exponential map $\exp: \hat{\Z}\times \C\rightarrow \C^{*}_\Q$ and consider a function:
$$\mu:\hat{\Z}\rightarrow L_\infty(\C)_1\ d\bar{z}\otimes \partial_z$$
such that $\mu(a+n)= T_{2\pi n}^{*}\left(\mu(a)\right)$. The \textsf{solenoidal Beltrami equation} associated to $\mu$ is a family of Beltrami equations:
$$\partial_{\bar{z}}f_a= \mu(a)\partial_z f_a$$
in the distributional sense. A \textsf{quasiconformal solution of the solenoidal Beltrami equation} is a continuous leaf preserving map $f: \hat{\Z}\times \C\rightarrow \hat{\Z}\times \C$ which descends by the exponential map to a continuous map $\hat{f}$ from the algebraic solenoid to itself and the map $z\mapsto f(a,z)$ is a quasiconformal solution of the Beltrami equation $\partial_{\bar{z}}f_a= \mu(a)\partial_z f_a$ for every adelic integer $a$. In particular, $f$ is a homeomorphism that is quasiconformal on every leaf and it descends to a leaf preserving homeomorphism $\hat{f}$ from the algebraic solenoid to itself.

A solution is determined by its restriction to the baseleaf which is a quasiconformal map $f_0$ whose expression is $f_0(z)= z+g(z)$ such that $g$ is a limit periodic respect to $x$ continuous map.

If there is a solution of the solenoidal Beltrami equation, then $\mu$ must be a continuous function on $\Z$ for $\mu(a)= \partial_{\bar{z}}f_a/\partial_z f_a$ and $f$ is continuous respect to the $a$ variable. We finally arrive at the definition:

\begin{defi}
A \textsf{solenoidal Beltrami differential} is a continuous function:
$$\mu:\hat{\Z}\rightarrow L_\infty(\C)_1\ d\bar{z}\otimes \partial_z$$
such that $\mu(a+n)= T_{2\pi n}^{*}\left(\mu(a)\right)$ for every integer $n$ and adelic integer $a$.
\end{defi}

A Beltrami differential $\mu$ is \textsf{limit periodic} if for every $\varepsilon>0$ there is an integer $N$ such that:
$$||T_{2\pi Nn}^{*}\left(\mu\right)- \mu||_\infty<\varepsilon$$
for every integer $n$. A solenoidal Beltrami differential $\hat{\mu}$ is limit periodic if $\hat{\mu}(a)$ is limit periodic for every adelic integer $a$.  Because of the structural condition, every solenoidal Beltrami differential is limit periodic and a limit periodic Beltrami differential on the baseleaf determines a unique solenoidal Beltrami differential.

Analogous definitions for the periodic case: A Beltrami differential $\mu$ is $2\pi n$--\textsf{periodic} if $T_{2\pi n}^{*}\left(\mu\right)= \mu$. A solenoidal Beltrami differential $\hat{\mu}$ is $2\pi n$--periodic if $\hat{\mu}(a)$ is $2\pi n$--periodic for every adelic integer $a$.

Trivially, solenoidal Beltrami equations associated to periodic solenoidal Beltrami differentials have solution. In effect, consider a $2\pi n$--periodic solenoidal Beltrami differential $\mu$. There is a $\mu_n$ Beltrami differential such that $\mu= \exp^{*}\circ\pi_n^{*}\left(\mu_n\right)$ where $\pi_n$ is the canonical projection of the algebraic solenoid as an inverse limit. Consider a quasiconformal normal solution $f_n$ of the Beltrami equation associated to $\mu_n$ and define $\hat{f}_n$ and $F_n$ as the unique leaf preserving maps such that the following diagram commutes:

$$\xymatrix{ 	\hat{\Z}\times \C \ar[rr]^{F_n} \ar[d]_{\exp} & & \hat{\Z}\times \C \ar[d]^{\exp} \\
			\C^{*}_\Q \ar[rr]^{\hat{f}_n} \ar[d]_{\pi_n} & & \C^{*}_\Q \ar[d]^{\pi_n} \\
			\C \ar[rr]^{f_{n}}  & & \C }$$

Then, $F_n$ is a quasiconformal solution of the solenoidal Beltrami equation associated to $\mu$. The solenoidal map $\hat{f}_n$ will be called a \textsf{finite solution}: It factors through some finite level of the inverse system.

Naively, we would expect that solenoidal Beltrami equations associated to limit periodic solenoidal Beltrami differentials would also have solution. However, as the next Lemma shows, this is not the case.

\begin{lema}\label{counterexample}
Consider the continuous limit periodic Beltrami differential $\mu$ such that:
$$\mu(z)= \frac{\frac{1}{2e}\sum_{n=1}^{+\infty} \left[ \cos(x/n!) -\frac{2iy}{n!}\sin(x/n!)\right] \frac{e^{-\frac{y^{2}}{n!^{2}}}}{2n!} } {1+ \frac{1}{2e}\sum_{n=1}^{+\infty} \left[ \cos(x/n!) +\frac{2iy}{n!}\sin(x/n!)\right] \frac{e^{-\frac{y^{2}}{n!^{2}}}}{2n!} }\ d\bar{z}\otimes \partial_z$$
where $z=x+iy$. There is a unique solenoidal Beltrami differential $\hat{\mu}$ such that $\hat{\mu}(a)$ is continuous for every adelic integer $a$ and $\hat{\mu}(0)=\mu$ whose associated solenoidal Beltrami equation has no solution.
\end{lema}
\begin{proof}
First, let's see that $\mu$ is a Beltrami differential; i.e. $||\mu||_{\infty}<1$. Define:
$$h_{\pm}(n)(z)=\frac{1}{2n!}e^{\pm ix/n!}\left(1\mp 2y/n!\right)\frac{e^{-\frac{y^{2}}{n!^{2}}}}{2}$$
where $z=x+iy$. Because $||h_{\pm}(n)||_{\infty}<1/2n!$ and the identity:
$$h_{+}(n) \pm h_{-}(n)= \left[ \cos(x/n!) \mp \frac{2iy}{n!}\sin(x/n!)\right] \frac{e^{-\frac{y^{2}}{n!^{2}}}}{2n!}$$
we have that each term of the sum has supremum norm less than $1/n!$ hence the supremum norm of the sum is less than $e-1$. We conclude that:
$$||\mu||_{\infty}< \frac{\frac{1}{2e}(e-1)}{1-\frac{1}{2e}(e-1)}=\frac{e-1}{e+1}<1/2$$

For the extension, define $\hat{\mu}(n):= T_{2\pi n}^{*}\left(\mu\right)$ for every integer $n$. Because $\mu$ is limit periodic respect to $x$, $\hat{\mu}$ uniquely extends to a continuous map on $\hat{\Z}$ such that $\hat{\mu}(a)$ is continuous for every adelic integer $a$.

The quasiconformal solution $w^{\mu}$ of the Beltrami equation associated to $\mu$ is the following:
$$w^{\mu}(z)= z+ \frac{1}{2e}\sum_{n=1}^{+\infty} \sin(x/n!)e^{-\frac{y^{2}}{n!^{2}}}$$
However, it is not of the type $z+ g(z)$ with $g$ limit periodic respect to $x$: It is the pointwise limit of periodic functions but the convergence is not uniform. In particular, its unique continuous extension to $\hat{\Z}\times\C$ verifying the structural condition does not descend to the algebraic solenoid.
\end{proof}

The previous solenoidal Beltrami differential is the pullback by the exponential of a continuous Beltrami differential on the algebraic solenoid $\C^{*}_\Q\subset \C P^{1}_\Q$ with vanishing limits at the cusps of the adelic Riemann sphere $\C P^{1}_\Q$. Actually, this is the purpose of the gaussian respect to the $y$ variable. However, even in this case there is no solution.

We seek for an additional condition on the solenoidal Beltrami differential to guarantee the existence of a solution of the associated solenoidal Beltrami equation.

\section{Statement and sketch of the proof}

\begin{teo}\label{Main}
Consider a solenoidal Beltrami differential $\mu$ with a net of periodic solenoidal Beltrami differential $(\mu_n)_{n\in\N}$ uniformly converging to it. Suppose there is a cofinal sequence $(n_i)_{i\in \N}$ in the divisibility net such that $n_i | n_{i+1}$ and the following series converges:
$$\sum_{i\in \N} n_{i+1}||\mu_{n_{i+1}}-\mu_{n_i}||_{\infty} <\infty$$

Then, there is a unique quasiconformal leaf preserving solution $\hat{f}$ from the adelic Riemann sphere to itself of the solenoidal Beltrami equation associated to $\mu$ such that $f$ fixes $0,1$ and $\infty$.
\end{teo}

We give a sketch of the proof. Although the details and estimations are rather technical and tedious, the idea of the proof is quite simple:
\begin{enumerate}
\item First consider a solenoidal Beltrami differential with compact support and consider the finite solutions $\hat{f}_n$ defined before.

\item For every natural $L$, the sequence $(\pi_{L}\circ\hat{f}_{n_{i}})_{i\in\N}$ is a uniform Cauchy sequence on compact sets of $\C_\Q$. In particular, there is a continuous function $g_{L}:\C_{\Q}\rightarrow \C$ such that the previous sequence converges to it on compact sets. By continuity, we have the relation $z^{L'/L}\circ g_{L'}= g_{L}$ for every $L'$ such that $L|L'$. By the universal property of inverse limits, there is a unique continuous leaf preserving function $\hat{f}:\C_{\Q}\rightarrow \C_{\Q}$ such that $\pi_{L}\circ \hat{f}= g_{L}$ for every natural $L$. By construction, the sequence $(\pi_{L}\circ\hat{f}_{n_{i}})_{i\in\N}$ converges locally uniformly to $\pi_{L}\circ \hat{f}$ for every natural $L$ hence the sequence $\left(\hat{f}_{n_{i}}\right)_{i\in \N}$ converges locally uniformly to $\hat{f}$.\label{step2}

\item The continuous map $\hat{f}$ is proper hence we get a continuous extension to the adelic Riemann sphere just defining $\hat{f}(\infty)=\infty$. On the adelic Riemann sphere, the sequence $\left(\hat{f}_{n_{i}}\right)_{i\in \N}$ converges pointwise to $\hat{f}$.\label{step3}

\item By standard arguments, we show that $\hat{f}$ is quasiconformal on every leaf and verifies the solenoidal Beltrami equation.

\item Finally, although tedious and lengthy, we remove the compact support hypothesis by a straightforward adaptation of the standard argument.

\end{enumerate}

Steps \ref{step2} and \ref{step3} are the heart of the proof. Here lies the real difficulty and the estimates are delicate and not trivial.

We can write the additional condition as the convergence of the following norm:
$$||\mu||_{\mathcal{S}}:= n_1||\mu_{n_1}||_{\infty} + \sum_{i\in \N} n_{i+1}||\mu_{n_{i+1}}-\mu_{n_i}||_{\infty}$$
where $\mathcal{S}$ denotes the cofinal sequence $(n_i)_{i\in \N}$. Because this new norm dominates the original $L_\infty$ one, the subspace of those differentials whose new norm is convergent is a Banach space with this norm and is the closure of the space of periodic differentials under this norm.

\section{The proof}

We will find a solution by approximating finite ones. As we explained before, the ramification points of the sphere must be fixed and this is no longer a choice but a topological constraint of the theory. Therefore we will consider normal quasiconformal solutions.

Consider a solenoidal Beltrami differential $\hat{\mu}$ with compact support and a net of periodic solenoidal Beltrami differentials $(\hat{\mu}_n)_{n\in\N}$ uniformly converging to it. Denote by $\mu$ and $\mathcal{I}_{n}(\mu)$ the pushouts of $\hat{\mu}$ and $\hat{\mu}_n$ respectively by the exponential map to the algebraic solenoid.

There are plenty of ways to construct this converging net. As an example, consider the following one: For every natural $n$, define the $2\pi n$--periodic Beltrami differential $\tilde{\mu}_{n}$ such that it equals $\hat{\mu}(0)$ on $[0, 2\pi n)\times \R$ and then extend it by the periodicity condition. There is a unique solenoidal Beltrami differential $\hat{\mu}_n$ such that $\hat{\mu}_n(0)=\tilde{\mu}_n$. Because $\hat{\mu}$ is limit periodic, the divisibility net $\left(\hat{\mu}_n\right)_{n\in\N}$ converges uniformly to $\hat{\mu}$.

For every integer $n$, define the Beltrami differential $\mu_n$ such that $\mathcal{I}_{n}(\mu)= \pi_{n}^{*}(\mu_{n})$. Consider the quasiconformal normal solution $f_{n}:\C\rightarrow\C$ of the $\mu_{n}$-Beltrami equation; i.e. $f_n(0)=0$ and $(f_n)_{z}-1\in L_{p}(\C)$, $p>2$. If $n|L$, define the map $f_{n}^{\uparrow L}$ and the leaf preserving solenoidal map $\hat{f}_{n}$ such that:

$$\xymatrix{ 	\C^{*}_\Q \ar[rr]^{\hat{f}_{n}} \ar[d]_{\pi_{L}} & & \C^{*}_\Q \ar[d]^{\pi_{L}} \\
			\C \ar[rr]^{f_{n}^{\uparrow L}} \ar[d]_{z^{L/n}} & & \C \ar[d]^{z^{L/n}} \\
			\C \ar[rr]^{f_{n}}  & & \C }$$

Define $\mu_{n}^{\uparrow L}:=(z^{L/n})^{*}(\mu_{n})$. The map $f_{n}^{\uparrow L}$ is a quasiconformal normal solution of the Beltrami equation associated to $\mu_{n}^{\uparrow L}$.

If $m|n|L$ define $f_{n,m}^{\uparrow L}= f_{n}^{\uparrow L}\circ (f_{m}^{\uparrow L})^{-1}$. See that $\hat{f}_{n,m}= \hat{f}_{n}\circ \hat{f}_{m}^{-1}$. The quasiconformal normal map $f_{n,m}^{\uparrow L}$ is the solution of the $\mu_{n,m}^{\uparrow L}$-Beltrami equation such that:
$$f_{m}^{*}(\mu_{n,m}^{\uparrow L})= \frac{\mu_{n}^{\uparrow L}- \mu_{m}^{\uparrow L}}{1- \mu_{n}^{\uparrow L}\overline{\mu_{m}^{\uparrow L}}}\ \overline{dz}\otimes \partial_z$$
where $\mu_{n}^{\uparrow L}$ and $\mu_{m}^{\uparrow L}$ on the right side denote the $L_\infty$ classes, the coefficients, and not the differentials.

\begin{lema}\label{BoundsLemma}
We have the following bounds:
\begin{enumerate}
\item $||\mu_{n}^{\uparrow L}||_{\infty} = ||\mathcal{I}_{n}(\mu)||_{\infty}$
\item $||\mu_{n,m}^{\uparrow L}||_{\infty} \leq ||\mathcal{I}_{n}(\mu) - \mathcal{I}_{m}(\mu)||_{\infty}\left(1-||\mathcal{I}_{n}(\mu)||_{\infty}||\mathcal{I}_{m}(\mu)||_{\infty}\right)^{-1}$
\end{enumerate}
and these are independent of $L$.
\end{lema}
\begin{proof}
\begin{enumerate}
\item This is immediate from the definition: $||\mu_{n}^{\uparrow L}||_{\infty}=||\mu_{n}||_\infty =  ||\mathcal{I}_{n}(\mu)||_{\infty}$ for every natural $L$ such that $n|L$.

\item We also have:
\begin{equation}\label{ineq1}
||\mu_{n,m}^{\uparrow L}||_{\infty} \leq \frac{||\mu_{n}^{\uparrow L}- \mu_{m}^{\uparrow L}||_{\infty}}{1- ||\mu_{n}^{\uparrow L}||_{\infty}||\mu_{m}^{\uparrow L}||_{\infty}} = \frac{||\mathcal{I}_{n}(\mu) - \mathcal{I}_{m}(\mu)||_{\infty}}{1-||\mathcal{I}_{n}(\mu)||_{\infty}||\mathcal{I}_{m}(\mu)||_{\infty}}
\end{equation}
for every $m|n|L$, where the last step follows from the following calculation:
\begin{eqnarray*}
\parallel \mu_{n}^{\uparrow L}- \mu_{m}^{\uparrow L}\parallel_{\infty} &=& \parallel \pi_{L}^{*}( \mu_{n}^{\uparrow L}- \mu_{m}^{\uparrow L})\parallel_{\infty} \\
&=& \parallel \pi_{L}^{*}((z^{L/n})^{*}\mu_{n}- (z^{L/m})^{*}\mu_{m})\parallel_{\infty} \\
&=& \parallel \pi_{L}^{*}((z^{L/n})^{*}\mu_{n})- \pi_{L}^{*}((z^{L/m})^{*}\mu_{m})\parallel_{\infty}  \\
&=& \parallel \pi_{n}^{*}(\mu_{n})- \pi_{m}^{*}(\mu_{m})\parallel_{\infty} \\
&=& \parallel \mathcal{I}_{n}(\mu)- \mathcal{I}_{m}(\mu) \parallel_{\infty}
\end{eqnarray*}
\end{enumerate}
\end{proof}

From now on, we will also assume the following \textsf{additional condition}: There is a cofinal sequence $(n_{i})_{i\in\N}$ in the divisibility net such that $n_i |n_{i+1}$ for every natural $i$ and the following series converges:
$$\sum_{j=1}^{\infty} n_{j+1}\ ||\mathcal{I}_{n_{j+1}}(\mu) - \mathcal{I}_{n_{j}}(\mu)||_{\infty} < \infty$$

Choose a real number $k$ such that $||\mu||_\infty < k<1$. Because of Lemma \ref{BoundsLemma}, by relabeling the sequence if necessary, we may suppose that:
$$||\mu_{n_i}^{\uparrow n_j}||_{\infty},\ ||\mu_{n_{i},n_{i-1}}^{\uparrow n_j}||_{\infty}<k$$
for every pair of natural numbers $i,j$ such that $i\leq j$. In particular, there is a $p>2$ such that the normalization in the statement of Theorem \ref{normal_solution} holds for all of the normal solutions of the Beltrami equations associated to these Beltrami differentials.

\begin{lema}\label{Bound1_aux}
Define $L=n_{J}$. There are constants $A$ and $A'$ depending only on $\mu$, $k$ and $p$ such that:
\begin{itemize}
\item If $i\leq J$ then
$$|\pi_{L}\circ\hat{f}_{n_{i}}(x)| \leq \left( 1+ A||\mathcal{I}_{n_{i}}(\mu)||_{\infty} \right)\ max\{\ 1,\ |\pi_{L}(x)|\ \}$$

$$|\pi_{L}\circ\hat{f}_{n_{i}}(x)| \leq \left( 1+ A'||\mathcal{I}_{n_{i}}(\mu)- \mathcal{I}_{n_{i-1}}(\mu)||_{\infty} \right)\ max\{\ 1,\ |\pi_{L}\circ\hat{f}_{n_{i-1}}(x)|\ \}$$

\item If $i>J$ then
$$|\pi_{L}\circ\hat{f}_{n_{i}}(x)|\leq (1+ A||\mathcal{I}_{L}(\mu)||_{\infty})e^{ \frac{A'}{L}\ \sum_{j=J}^{i-1} n_{j+1}\ ||\mathcal{I}_{n_{j+1}}(\mu)- \mathcal{I}_{n_{j}}(\mu)||_{\infty}}\ max\{\ 1,\ |\pi_{L}(x)|\ \}$$
\end{itemize}
\end{lema}
\begin{proof}
Because the supports of all $\mu_{n_{i}}^{\uparrow L}$ and $\mu_{n_{i},n_{i-1}}^{\uparrow L}$ are uniformly bounded:
$$supp\big(\mu_{n_{i}}^{\uparrow L}\big),\ supp\big(\mu_{n_{i},n_{i-1}}^{\uparrow L}\big)\subset \pi_{1}\big(support(\mu)\big)\cup \overline{D(0;1)}$$
by Lemma \ref{Cota1_Prel} we can take the same constant $A$ for all the maps $f_{n_{i}}^{\uparrow L}$ and $f_{n_{i},n_{i-1}}^{\uparrow L}$. Then we have:
\begin{eqnarray*}
|\pi_{L}\circ\hat{f}_{n_{i}}(x)| &=& |f_{n_{i}}^{\uparrow L}(\pi_{L}(x))| \\
&\leq& A\ ||\mu_{n_{i}}^{\uparrow L}||_{\infty}\ |\pi_{L}(x)|^{1-2/p} + |\pi_{L}(x)| \\
&\leq& A\ ||\mathcal{I}_{n_{i}}(\mu)||_{\infty}\ |\pi_{L}(x)|^{1-2/p} + |\pi_{L}(x)| \\
&\leq& \left( 1+ A||\mathcal{I}_{n_{i}}(\mu)||_{\infty} \right)\ max\{\ 1,\ |\pi_{L}(x)|\ \}
\end{eqnarray*}
In particular,
\begin{equation}\label{aux_eq_1}
|\pi_{L}\circ\hat{f}_{L}(x)| \leq (1+ A||\mathcal{I}_{L}(\mu)||_{\infty})\ max\{\ 1,\ |\pi_{L}(x)|\ \}
\end{equation}

For the second assertion:
\begin{eqnarray*}
|\pi_{L}\circ\hat{f}_{n_{i},n_{i-1}}(x)| &=& |f_{n_{i},n_{i-1}}^{\uparrow L}(\pi_{L}(x))| \\
&\leq& A'\ ||\mathcal{I}_{n_{i}}(\mu) - \mathcal{I}_{n_{i-1}}(\mu)||_{\infty}\ |\pi_{L}(x)|^{1-2/p} + |\pi_{L}(x)| \\
&\leq& (1+ A'\ ||\mathcal{I}_{n_{i}}(\mu) - \mathcal{I}_{n_{i-1}}(\mu)||_{\infty})\ max\{\ 1,\ |\pi_{L}(x)|\ \}
\end{eqnarray*}
where $A'= A/(1-k^{2})$. Because $\hat{f}_{n_{i}}= \hat{f}_{n_{i},n_{i-1}}\circ \hat{f}_{n_{i-1}}$ the result follows.

Finally, for the third assertion we have:
\begin{eqnarray*}
|\pi_{n_{j}}\circ\hat{f}_{n_{j},n_{j-1}}(x)| &=& |f_{n_{j},n_{j-1}}(\pi_{n_{j}}(x))| \\
&\leq& A'\ ||\mathcal{I}_{n_{j}}(\mu) - \mathcal{I}_{n_{j-1}}(\mu)||_{\infty}\ |\pi_{n_{j}}(x)|^{1-2/p} + |\pi_{n_{j}}(x)| \\
&\leq& \left( \frac{A'a_{n_{j}}}{n_{j}}+1 \right)\ max\{\ 1,\ |\pi_{n_{j}}(x)|\ \}
\end{eqnarray*}
where $a_{n_{j}}= n_{j}\ ||\mathcal{I}_{n_{j}}(\mu) - \mathcal{I}_{n_{j-1}}(\mu)||_{\infty}$.
Because $\pi_{n_{j}}^{n_{j}/L}= \pi_{L}$ we have:
\begin{eqnarray*}
|\pi_{L}\circ\hat{f}_{n_{j},n_{j-1}}(x)| &=& |\pi_{n_{j}}\circ\hat{f}_{n_{j},n_{j-1}}(x)|^{n_{j}/L}\\
&\leq& \left( \frac{A'a_{n_{j}}}{n_{j}}+1 \right)^{n_{j}/L}\ max\{\ 1,\ |\pi_{n_{j}}(x)|\ \}^{n_{j}/L} \\
&\leq& e^{\frac{A'}{L}a_{n_{j}}}   \ max\{\ 1,\ |\pi_{L}(x)|\ \}
\end{eqnarray*}
In particular, because the right hand side of the above equation is greater than or equal to one, then:
\begin{equation}\label{aux_eq_2}
max\{\ 1,\ |\pi_{L}\circ\hat{f}_{n_{j},n_{j-1}}(x)|\ \} \leq e^{\frac{A'}{L}a_{n_{j}}}   \ max\{\ 1,\ |\pi_{L}(x)|\ \}
\end{equation}
and by the same argument, relation \eqref{aux_eq_1} implies:

\begin{equation}\label{aux_eq_3}
max\{\ 1,\ |\pi_{L}\circ\hat{f}_{L}(x)|\ \} \leq (1+ A||\mathcal{I}_{L}(\mu)||_{\infty})\ max\{\ 1,\ |\pi_{L}(x)|\ \}
\end{equation}

\noi Because $$\hat{f}_{n_{i}}= \hat{f}_{n_{i},n_{i-1}}\circ\hat{f}_{n_{i-1},n_{i-2}}\ldots \circ\hat{f}_{n_{J+1},n_{J}}\circ\hat{f}_{L}$$
induction on relation \eqref{aux_eq_2} and relation \eqref{aux_eq_3} imply:
$$|\pi_{L}\circ\hat{f}_{n_{i}}(x)| \leq max\{\ 1,\ |\pi_{L}\circ\hat{f}_{n_{i}}(x)|\ \} \leq (1+ A||\mathcal{I}_{L}(\mu)||_{\infty})e^{ \frac{A'}{L}\ \sum_{j=J}^{i-1}a_{n_{j+1}}}\ max\{\ 1,\ |\pi_{L}(x)|\ \}$$
and the result follows.
\end{proof}

\begin{cor}\label{Bound_ML}
Define $L=n_{J}$. There are constants $A$ and $A'$ depending only on $\mu$, $k$ and $p$ such that for every natural $i\geq J$:
$$|\pi_{L}\circ\hat{f}_{n_{i}}(x)|\leq (1+ A||\mathcal{I}_{L}(\mu)||_{\infty})e^{ \frac{A'}{L}\ \sum_{j=J}^{\infty} n_{j+1}\ ||\mathcal{I}_{n_{j+1}}(\mu)- \mathcal{I}_{n_{j}}(\mu)||_{\infty}}\ max\{\ 1,\ |\pi_{L}(x)|\ \}$$
\end{cor}
\begin{proof}
The convergence of the series implies the hypothesis of the previous Lemma \ref{Bound1_aux}. Taking the limit $i\to\infty$ on the right hand side of the relation gives the result.
\end{proof}

\begin{lema}\label{Bound_ML2}
Define $L=n_{J}$. There are constants $B$ and $B'$ depending only on $\mu$, $k$ and $p$ such that for every natural $i\geq J$:
$$|\pi_{L}(x)|\leq (1+ B||\mathcal{I}_{L}||_{\infty})e^{ \frac{B'}{L}\ \sum_{j=J}^{\infty}n_{j+1}\ ||\mathcal{I}_{n_{j+1}}(\mu)- \mathcal{I}_{n_{j}}(\mu)||_{\infty}}\ max\{\ 1,\ |\pi_{L}\circ\hat{f}_{n_{i}}(x)|\ \}$$
\end{lema}
\begin{proof}
The proof is almost verbatim to the proof of Lemma \ref{Bound1_aux} with reference to Lemma \ref{Cota2_prel} instead of \ref{Cota1_Prel}.
\end{proof}

\begin{lema}\label{TheFormula}
Define $L=n_{J}$. There is a constant $A'$ depending only on $\mu$, $k$ and $p$ and a constant $M_{L}\geq 1$ such that for every $i\geq J$:
$$| \pi_{L}\circ \hat{f}_{n_{i+1}}(x) - \pi_{L}\circ \hat{f}_{n_{i}}(x)| \leq \frac{A'}{L}\ n_{i+1}\ ||\mathcal{I}_{n_{i+1}}(\mu)-\mathcal{I}_{n_{i}}(\mu)||_{\infty}\ M_{L}\ max\{\ 1,\ |\pi_{L}(x)|\ \}$$
\end{lema}
\begin{proof}
We take the same values of $k<1$, $p>2$ and constants $A$ and $A'= A/(1-k^{2})$ as those in the proof of Lemma \ref{Bound1_aux}. Denote $n=n_{i}$ and $m=n_{i-1}$. By Lemma \ref{Cota1_Prel} and relation \eqref{ineq1} we have:
\begin{equation}\label{des_no_renorm}
|f_{n,m}(\pi_{n}(x))- \pi_{n}(x)|\leq A' \parallel \mathcal{I}_{n}(\mu)- \mathcal{I}_{m}(\mu) \parallel_{\infty} \mid \pi_{n}(x) \mid ^{1-2/p}
\end{equation}
where $A'= A/(1-k^{2})$ and $\parallel \mu \parallel_{\infty} = k$. Define $n'=n/L$. Lagrange Theorem implies:
\begin{eqnarray}\label{Lagrange}
\nonumber | \pi_{L}\circ \hat{f}_{n,m}(x) - \pi_{L}(x)| &=& | f_{n,m}(\pi_{n}(x))^{n'} - \pi_{L}(x)| \\
&\leq& n'\ |\xi|^{n'-1}\ |f_{n,m}(\pi_{n}(x)) - \pi_{n}(x)|
\end{eqnarray}

\noi for some $\xi$ in the interior of the segment joining $\pi_{n}\circ \hat{f}_{n,m}(x)$ and $\pi_{n}(x)$. In particular,
\begin{eqnarray}\label{bound}
\nonumber |\xi|^{n'-1} &\leq&  max\{\ |\pi_{n}(x)|,\ |\pi_{n}\circ \hat{f}_{n,m}(x)|\ \}^{n'-1} \\
&=&  max\{\ |\pi_{L}(x)|,\ |\pi_{L}\circ \hat{f}_{n,m}(x)|\ \}^{1-1/n'}
\end{eqnarray}

\noi Equations \eqref{des_no_renorm}, \eqref{Lagrange} and \eqref{bound} imply:
\begin{eqnarray}
\nonumber && | \pi_{L}\circ \hat{f}_{n,m}(x) - \pi_{L}(x)| \leq \frac{A'}{L}\ n\ ||\mathcal{I}_{n}(\mu)-\mathcal{I}_{m}(\mu)||_{\infty}\ldots \\
&& \ldots max\{\ |\pi_{L}(x)|,\ |\pi_{L}\circ \hat{f}_{n,m}(x)|\ \}^{1-1/n'} \mid \pi_{L}(x) \mid ^{(1-2/p)1/n'}
\end{eqnarray}

\noi In particular, because $\hat{f}_{n}= \hat{f}_{n,m}\circ \hat{f}_{m}$ we have:
\begin{eqnarray}\label{des_renorm}
\nonumber && | \pi_{L}\circ \hat{f}_{n}(x) - \pi_{L}\circ \hat{f}_{m}(x)| \leq \frac{A'}{L}\ n\ ||\mathcal{I}_{n}(\mu)-\mathcal{I}_{m}(\mu)||_{\infty} \ldots \\
&& \ldots max\{\ |\pi_{L}\circ \hat{f}_{m}(x)|,\ |\pi_{L}\circ \hat{f}_{n}(x)|\ \}^{1-1/n'} \mid \pi_{L}\circ \hat{f}_{m}(x) \mid ^{(1-2/p)1/n'}
\end{eqnarray}

\noi By the previous corollary \ref{Bound_ML} there is a constant $M_{L}$ such that:
\begin{equation}
|\pi_{L}\circ\hat{f}_{n_{i}}(x)| \leq M_{L}\ max\{\ 1,\ |\pi_{L}(x)|\ \}
\end{equation}
for every $i\geq J$ where $L=n_{J}$. This bound implies:
\begin{eqnarray}
\nonumber && | \pi_{L}\circ \hat{f}_{n_{i+1}}(x) - \pi_{L}\circ \hat{f}_{n_{i}}(x)| \\
\nonumber && \leq \frac{A'}{L}\ n_{i+1}\ ||\mathcal{I}_{n_{i+1}}(\mu)-\mathcal{I}_{n_{i}}(\mu)||_{\infty}\ (M_{L}\ max\{\ 1,\ |\pi_{L}(x)|\ \})^{1-2/pn'} \\
&& \leq \frac{A'}{L}\ n_{i+1}\ ||\mathcal{I}_{n_{i+1}}(\mu)-\mathcal{I}_{n_{i}}(\mu)||_{\infty}\ M_{L}\ max\{\ 1,\ |\pi_{L}(x)|\ \}
\end{eqnarray}
where we have used that $M_{L}\ max\{\ 1,\ |\pi_{L}(x)|\ \} \geq 1$ and the formula is proved.
\end{proof}

\begin{lema}\label{existence_continuous_map}
There is a continuous leaf preserving map $\hat{f}$ from the adelic Riemann sphere to itself fixing the cusps $0$ and $\infty$ such that $(\hat{f}_{n_{i}})_{i\in \N}$ converges pointwise to $\hat{f}$.
\end{lema}
\begin{proof}
\textit{Existence of $\hat{f}$:} For each $L=n_{J}$, Lemma \ref{TheFormula} implies that $(\pi_{L}\circ\hat{f}_{n_{i}})_{i\in\N}$ is a uniform Cauchy sequence on compact sets so there is a continuous function $g_{L}:\C_{\Q}\rightarrow \C$ such that the sequence $(\pi_{L}\circ\hat{f}_{n_{i}})_{i\in\N}$ converges uniformly to $g_{L}$ on compact sets. Consider another $L'= n_{J'}$ such that $J'>J$. Because $z^{L'/L}\circ \pi_{L'}\circ\hat{f}_{n_{i}}= \pi_{L}\circ\hat{f}_{n_{i}}$ for every $i\geq J'$ and the continuity of $z^{L'/L}$ we have that $z^{L'/L}\circ g_{L'}= g_{L}$. By the universal property of inverse limits there is a unique function $\hat{f}:\C_{\Q}\rightarrow \C_{\Q}$ such that $\pi_{n_{i}}\circ \hat{f}= g_{n_{i}}$ for every natural $i$. Because every $g_{n_{i}}$ is continuous we have that $\hat{f}$ is continuous and verifies that $(\pi_{L}\circ\hat{f}_{n_{i}})_{i\geq J}$ converges uniformly to $\pi_{L}\circ\hat{f}$ on compact sets. In particular, $(\hat{f}_{n_{i}})_{i\in \N}$ converges pointwise to $\hat{f}$. By definition every $\hat{f}_{n_{i}}$ is leaf preserving hence so is $\hat{f}$.

\textit{The map $\hat{f}$ is proper:} Consider a compact set $K\subset \hat{\C}_{\Q}$. The compact $K$ is closed for every compact subset of a Hausdorff space is closed and because $\hat{f}$ is continuous, $\hat{f}^{-1}(K)$ is closed. By Lemma \ref{Bound_ML2} and the fact that $(\hat{f}_{n_{i}})_{i\in \N}$ converges pointwise to $\hat{f}$, we have that for every $L=n_{J}$ there is a constant $M_{L}'$ such that:
$$|\pi_{L}(x)|\leq M_{L}'|\pi_{L}\circ\hat{f}(x)|$$
Choose some natural $L=n_{J}$. Define $R$ such that $d(0,\pi_{L}(K))=R<\infty$ for $\pi_{L}$ is continuous; i.e. $\pi_{L}(K)$ is compact. By the above relation we have that $$d(0,\pi_{L}(\hat{f}^{-1}(K)))\leq M_{L}'R$$ and because $\pi_{L}$ is proper, the closed set $\hat{f}^{-1}(K)$ is contained in the compact $\pi_{L}^{-1}(D(0;M_{L}'R))$ hence $\hat{f}^{-1}(K)$ is compact and we have the claim.

\textit{Extension:} In particular, the extension $\hat{f}:\C P^{1}_\Q\rightarrow \C P^{1}_\Q$ such that $\hat{f}(\infty)= \infty$ is continuous and because $\hat{f}_{n_{i}}(\infty)= \infty$ for every natural $i$, we have that $(\hat{f}_{n_{i}})_{i\in \N}$ converges pointwise to $\hat{f}$ on $\C P^{1}_\Q$. This finishes the proof.
\end{proof}

\bigskip

\noi\textbf{Proof of Theorem \ref{Main}:}

\noi\textit{(Uniqueness)} Suppose that $f$ and $g$ are solenoidal quasiconformal solutions to the solenoidal Beltrami equation associated to $\mu$ fixing $0,1,\infty$. Then, $f\circ g^{-1}$ is a leaf preserving $1$-quasiconformal map fixing $0,1,\infty$. There is a holomorphic limit periodic respect to $x$ function $h$ such that $\nu^{-1}\circ f\circ g^{-1}\circ \nu(z)= z+h(z)$ where $\nu$ is the baseleaf. On the other hand, by Weyl's Lemma $\nu^{-1}\circ f\circ g^{-1}\circ \nu$ is a holomorphic homeomorphism of $\C$; i.e. an affine transformation. Because it fixes zero, we have that $\nu^{-1}\circ f\circ g^{-1}\circ \nu=id$ hence $f\circ g^{-1}=id$ and $f=g$.

\noi\textit{(Existence)} First suppose that $\mu$ has compact support in $\C_{\Q}$. Consider an arbitrary leaf $\nu:\C\rightarrow \C^{*}_{\Q}\subset\C P^{1}_\Q$. By Lemma \ref{existence_continuous_map} there is a continuous leaf preserving map $\hat{f}:\C P^{1}_\Q\rightarrow \C P^{1}_\Q$ such that $(\hat{f}_{n_{i}})_{i\in \N}$ converges pointwise to $\hat{f}$. In particular, the sequence $(\nu^{-1}\circ\hat{f}_{n_{i}}\circ\nu)_{i\in\N}$ converges pointwise to $\nu^{-1}\circ\hat{f}\circ\nu$. By the analytic definition of quasiconformal maps (section 4.1.2 in \cite{Imayoshi}), because $\pi_n\circ \nu= \exp(2\pi i n z)$ is a holomprphic map, the maps $\nu^{-1}\circ\hat{f}_{n_{i}}\circ\nu$ are quasiconformal solutions of the respectives $\nu^{*}(\mathcal{I}_{n_{i}}(\mu))$-Beltrami equations.

To use the standard convergence Theorem for quasiconformal maps, we need to change the normalization: Define the affine maps $A_{i}(z)= a_{i}z+b_{i}$ such that $A_{i}^{-1}\circ\nu^{-1}\circ\hat{f}_{n_{i}}\circ\nu$ is the quasiconformal solution of the $\nu^{*}(\mathcal{I}_{n_{i}}(\mu))$-Beltrami equation fixing $0,1,\infty$ (See remark \ref{TradeOff} below). Concretely:
$$a_{i}= \nu^{-1}\circ\hat{f}_{n_{i}}\circ\nu(1)- \nu^{-1}\circ\hat{f}_{n_{i}}\circ\nu(0)$$
$$b_{i}= \nu^{-1}\circ\hat{f}_{n_{i}}\circ\nu(0)$$
Because $(\hat{f}_{n_{i}})_{i\in \N}$ converges pointwise to $\hat{f}$, the sequence of affine maps $(A_{i})_{i\in\N}$ converges locally uniformly to the map $A(z)= az+b$ such that:
$$a = \nu^{-1}\circ\hat{f} \circ\nu(1)- \nu^{-1}\circ\hat{f}\circ\nu(0)$$
$$b = \nu^{-1}\circ\hat{f} \circ\nu(0)$$
A priori $a$ could be zero. Define the map $g$ as the quasiconformal solution of the $\nu^{*}(\mu)$-Beltrami equation fixing $0,1,\infty$. Because $\mathcal{I}_{n_{i}}(\mu)$ tends to $\mu$ in $L_{\infty}(\C_{\Q})$ we have that $\nu^{*}(\mathcal{I}_{n_{i}}(\mu))$ tends to $\nu^{*}(\mu)$ in $L_{\infty}(\C)$ and by Lemma \ref{Imayoshi_Taniguchi_Convergence_infty} we conclude that:
$$A_{i}^{-1}\circ\nu^{-1}\circ\hat{f}_{n_{i}}\circ\nu\xrightarrow{i\to\infty} g$$
locally uniformly. Then:
$$\nu^{-1}\circ\hat{f}_{n_{i}}\circ\nu\xrightarrow{i\to\infty} A\circ g$$
and we conclude that:
$$\nu^{-1}\circ\hat{f} \circ\nu= A\circ g$$
Because $\hat{f}$ is continuous and fixes $0,\infty$ it cannot be constant. In particular $a\neq 0$ and we have that $\nu^{-1}\circ\hat{f} \circ\nu$ is a quasiconformal solution of the $\nu^{*}(\mu)$-Beltrami equation for every leaf $\nu$. Finally, $\hat{f}$ is a homeomorphism for every continuous bijective map between compact sets is a homeomorphism. We have proved that $\hat{f}$ is quasiconformal homeomorphism. Multiplying by $\hat{f}(1)^{-1}$ we have the quasiconformal solution fixing $0,1,\infty$.

Now we remove the hypothesis of the compact support of $\mu$ by the standard well known trick: Define $\mu_{1}= \mu.\chi_{|\pi_{1}(z)|\geq 1}$ and consider the M\"obius inversion $\gamma:\C P^{1}_\Q\rightarrow \C P^{1}_\Q$ such that $\gamma(z)= z^{-1}$. Because $\gamma^{*}(\mu_{1})$ has compact support on $\C_{\Q}$, by the previous part there is a unique quasiconformal leaf preserving solution $g:\C P^{1}_\Q\rightarrow \C P^{1}_\Q$ to the $\gamma^{*}(\mu_{1})$-Beltrami equation such that $g$ fixes $0,1,\infty$. Define $f_{1}$ such that the following diagram commutes: 
$$\xymatrix{ 	\C P^{1}_\Q \ar[rr]^{g} \ar[d]_{\gamma} & & \C P^{1}_\Q \ar[d]^{\gamma} \\
			\C P^{1}_\Q \ar[rr]^{f_{1}}  & & \C P^{1}_\Q  }$$
			
\noi \textit{Claim:} The map $f_{1}$ is the quasiconformal solution of the $\mu_{1}$-Beltrami equation fixing $0,1,\infty$: Because $\gamma$ and $g$ are homeomorphisms fixing $0,1,\infty$ so is $f_{1}$. For every leaf $\nu_{a}$ we have the diagram:

$$\xymatrix{
\C \ar[rrrr]^(.3){\nu_{-a}^{-1}\circ g\circ\nu_{-a}} \ar[dd]^(.7){-z} \ar@{^{(}->}[rrd]^{\nu_{-a}}	&	& &		& \C \ar'[d][dd] \ar@{}[dd]^(.7){-z} \ar@{^{(}->}[rrd]^{\nu_{-a}}	&	&	\\
		& 		& \C P^{1}_\Q \ar[rrrr]^(.3){g} \ar[dd]^(.3){\gamma}	&	& &			& \C P^{1}_\Q \ar[dd]^(.3){\gamma}	\\
    \C \ar|{\ \ }[rrrr] \ar@{^{(}->}[rrd]^{\nu_{a}}	\ar@{}[rrrr]^(.3){\nu_{a}^{-1}\circ f_{1}\circ\nu_{a}}	&	& &		& \C \ar@{^{(}->}[rrd]^{\nu_{a}}			&	&	\\
		&		& \C P^{1}_\Q  \ar[rrrr]^(.3){f_{1}}		&	& &			& \C P^{1}_\Q		\\}$$

\noi Because every $\nu_{a}$ is injective and the left, right, top, bottom and front sides commute we have that the back face also commutes. By definition $\nu_{-a}^{-1}\circ g\circ\nu_{-a}$ is a quasiconformal solution of the $\nu_{-a}^{*}\circ\gamma^{*}(\mu_{1})$-Beltrami equation so $\nu_{a}^{-1}\circ f_{1}\circ\nu_{a}$ is a quasiconformal solution of the $(-z)^{*}\circ\nu_{-a}^{*}\circ\gamma^{*}(\mu_{1})$-Beltrami equation. We have:
$$(-z)^{*}\circ\nu_{-a}^{*}\circ\gamma^{*}(\mu_{1})= (\gamma\circ\nu_{-a}\circ (-z))^{*}(\mu_{1})= \nu_{a}^{*}(\mu)$$
and this proves the claim.

Define the adelic differential $\mu_{2}$ such that:
$$f_{1}^{*}(\mu_{2})= \frac{\mu-\mu_{1}}{1-\mu\overline{\mu_{1}}}\ \overline{d\pi_{1}}\otimes (d\pi_{1})^{-1}$$
On the right hand side we denote by $\mu$ and $\mu_1$ the $L_\infty$ classes, the coefficients, and not the differentials. We have:

$$\nu_{a}^{*}(\mu) = (\mu\circ\nu_{a})\ \frac{\overline{(\pi_{1}\circ \nu_{a})'}} {(\pi_{1}\circ \nu_{a})'} = (\mu\circ\nu_{a})\ \frac{\overline{(e^{iz})'}}{(e^{iz})'} = -e^{-i(z+\bar{z})}(\mu\circ\nu_{a})= \mu_{a}$$
and a similar expression and definition for $\nu_{a}^{*}(\mu_{1})$:
$$\nu_{a}^{*}(\mu_{1})= -e^{-i(z+\bar{z})}(\mu_{1}\circ\nu_{a}) = \mu_{1,a}$$
\noi A similar calculation gives:
$$(f_{1}\circ\nu_{a})^{*}(\mu_{2}) = \nu_{a}^{*}(f_{1}^{*}(\mu_{2}))= -e^{-i(z+\bar{z})}\left(\frac{\mu-\mu_{1}}{1-\mu\overline{\mu_{1}}}\right)\circ\nu_{a} = \frac{\mu_{a}-\mu_{1,a}}{1-\mu_{a}\overline{\mu_{1,a}}}$$

\noi Because $\mu_{2}$ has compact support on $\C_{\Q}$ there is a unique quasiconformal leaf preserving solution $f_{2}$ to the $\mu_{2}$-Beltrami equation fixing $0,1,\infty$. Define the map $f=f_{2}\circ f_{1}$. It is clearly quasiconformal leaf preserving and fixes $0,1,\infty$ for it is the composition of maps of the same kind. Because:
$$\nu_{a}^{-1}\circ f\circ\nu_{a}= \nu_{a}^{-1}\circ f_{2}\circ f_{1}\circ\nu_{a}= (\nu_{a}^{-1}\circ f_{2}\circ\nu_{a})\circ (\nu_{a}^{-1}\circ f_{1}\circ\nu_{a})$$
and the following fact:
$$(\nu_{a}^{-1}\circ f_{1}\circ\nu_{a})^{*}(\nu_{a}^{*}(\mu_{2}))= (f_{1}\circ\nu_{a})^{*}(\mu_{2})= \frac{\mu_{a}-\mu_{1,a}}{1-\mu_{a}\overline{\mu_{1,a}}}$$
we conclude that $\nu_{a}^{*}(\mu)= \mu_{a}$ is the Beltrami differential of $\nu_{a}^{-1}\circ f\circ\nu_{a}$ for every leaf $\nu_{a}$; i.e. $f$ is the unique quasiconformal solution to the $\mu$-Beltrami equation fixing $0,1,\infty$.
\fdem

\begin{obs}\label{TradeOff}
At first sight it seems there is something terribly wrong in the above proof: While $\hat{f}$ fixes $0,\infty$ and has only one degree of freedom as a solution of the $\mu$-Beltrami equation, its conjugated map $\nu^{-1}\circ\hat{f}\circ\nu$ has two degrees of freedom. Why the conjugated map has an extra degree of freedom? Let's see: The conjugated map has the same freedom as $\hat{f}$ plus the property of being uniformly limit periodic on horizontal bands. Once this last property is destroyed by an affine transformation, an extra degree of freedom comes out.
\end{obs}

\end{document}